\documentclass[article]{amsart}
\usepackage{cases}
\usepackage{amsmath, amsfonts, pifont, amssymb}
\usepackage{verbatim}
\usepackage[bookmarks=true]{hyperref}
\usepackage{geometry}
\newtheorem{theorem}{Theorem}[section]
\newtheorem{lemma}[theorem]{Lemma}
\newtheorem{proposition}[theorem]{Proposition}

\newcommand{\beq}{\begin{equation}}
\newcommand{\eeq}{\end{equation}}
\newcommand{\beqq}{\begin{equation*}}
\newcommand{\eeqq}{\end{equation*}}

\theoremstyle{definition}
\newtheorem{definition}[theorem]{Definition}

\theoremstyle{remark}
\newtheorem{remark}[theorem]{Remark}

\numberwithin{equation}{section}

\begin{document}

\title[Well-posedness for NLS on waveguide]{Well-posedness for energy-critical nonlinear Schr\"odinger equation on waveguide manifold}
\author{Xing Cheng, Zehua Zhao, and Jiqiang Zheng}
\maketitle

\begin{abstract}
In this article, we utilize the scale-invariant Strichartz estimate on waveguide which is developed recently by Barron \cite{Barron} based on Bourgain-Demeter $l^2$ decoupling method \cite{BD} to give a unified and simpler treatment of well-posedness results for energy critical nonlinear Schr\"odinger equation on waveguide when the whole dimension is three and four.
\end{abstract}
\bigskip

\bigskip


\section{introduction}

In this article, we study the well-posedness theory of the following nonlinear Schr\"odinger equation (NLS) on product space $\mathbb{R}^{m} \times \mathbb{T}^{n}$, where $m,n \in \mathbb{N}$,
\begin{equation}\label{maineq1}
\aligned
\begin{cases}
\left(i\partial_t+ \Delta_{\mathbb{R}^{m} \times \mathbb{T}^{n}}\right) u = F(u) = \pm |u|^{p} u, \\
u(0,x) = u_{0} \in H^{1}(\mathbb{R}^{m} \times \mathbb{T}^{n}).
\end{cases}
\endaligned
\end{equation}
Here the product space $\mathbb{R}^{m} \times \mathbb{T}^{n}$ is known as `semiperiodic space' as well as `waveguide manifold', where $\mathbb{T}^{n}$ is
a (rational or irrational) $n$-dimensional torus. Similar as the Euclidean case, we call \eqref{maineq1} `energy-critical' when the whole dimension $d=m+n$ and the exponent $p$ satisfies $p=\frac{4}{d-2}$. In this article, we consider equation \eqref{maineq1} with $3$ and $4$ whole dimensions. More precisely, we consider
\begin{equation}\label{maineq2}
\aligned
\begin{cases}
\left(i\partial_t+ \Delta_{\mathbb{R}^{m} \times \mathbb{T}^{n}}\right) u =  \pm |u|^{\frac{4}{m+n-2}} u, \\
u(0,x) = u_{0} \in H^{1}(\mathbb{R}^{m} \times \mathbb{T}^{n}),
\end{cases}
\endaligned
\end{equation}
where whole dimension $d=m+n=\{3,4\}$. There are exactly $9$ specific NLS models included. When $n=0$, the problems coincide NLS on pure Euclidean space; when $m=0$, the problems coincide NLS on pure tori.\vspace{3mm} 

Initial value problem \eqref{maineq2} is called defocusing if the sign of the nonlinearity is positive; focusing if the sign is negative. There are three important conserved quantities of \eqref{maineq2} as follows.
\begin{align*}
\text{mass: }    &\quad
{M}(u(t))  = \int_{\mathbb{R}^{m} \times \mathbb{T}^{n}} |u(t,x,y)|^2\,\mathrm{d}x\mathrm{d}y,\\
\text{   energy:  }     &  \quad
{E}(u(t))  = \int_{\mathbb{R}^{m} \times \mathbb{T}^{n}} \frac12 |\nabla u(t,x,y)|^2  \pm \frac{d-2}{2d} |u(t,x,y)|^{\frac{2d}{d-2}} \,\mathrm{d}x\mathrm{d}y,\\
\text{ momentum: } &  \quad
{P}(u(t)) = \Im \int_{\mathbb{R}^{m} \times \mathbb{T}^{n}} \overline{u(t,x,y)} \nabla u(t,x,y)\,\mathrm{d}x\mathrm{d}y.
\end{align*}
Well-posedness theory and long time behavior of NLS is a hot topic in the area of dispersive evolution equations and has been studied widely in recent decades. The Euclidean case is first treated and the theory at least in the defocusing setting is well established. We refer to \cite{CW,Iteam1,BD3,BDec4,KM1,KVnotes,Taobook} for some important Euclidean results. Moreover, we refer to \cite{HTT1,IPT3,KV1,Haitian} for the tori case. We may roughly think that the waveguide case is `between' the Euclidean case and the tori case in some sense since the waveguide is the product of the Euclidean space and the tori. Both of the techniques for the two cases are often combined and applied to the waveguide case. At last, we refer to \cite{IPS,PTW} regarding some results for NLS on other spaces such as hyperbolic space and sphere and refer to \cite{Cbook,BDbook,KVnotes,Taobook} for some classical textbooks or notes on dispersive PDE. \vspace{3mm}

The main result of this article is as follows.
\begin{theorem}[Local well-posedness] \label{main}
For $d=m+n=3,4$, we consider the initial value problem \eqref{maineq2}. Then there exists a time $T=T(u_0)$ and a unique solution $u \in C_t^0 ([0,T);H^1(\mathbb{R}^{m} \times \mathbb{T}^{n}) ) \cap X^1([0,T))$. Moreover, there exists $\eta_0=\eta_0(d)>0$ such that if $||u_0||_{H^1(\mathbb{R}^{m} \times \mathbb{T}^{n})}<\eta_0$, then the solution is global in time. 
\end{theorem}

\begin{remark}
The solution space $X^1$ is first introduced in Herr-Tataru-Tzvetkov \cite{HTT1,HTT2} which is based on atomic space and variation space. See Section 2 for explicit discussions.
\end{remark}

\begin{remark}\label{1.3}
This result is indifferent to the rational/irrational choice for the tori direction since the tools (such as Function spaces and Strichartz estimate) we use work for both cases. It is shown that irrational tori enjoy better Strichartz estimates on long time intervals (See Deng-Germain-Guth \cite{DGG} for more details).
\end{remark}

\begin{remark}
We expect to treat the high dimensional analogue of Theorem \ref{main} based on delicate treatment of the nonlinearity. One main concern would be that such function spaces are not well compatible with non-integer nonlinearity. In this paper, we discuss the cubic and the quintic nonlinearity as examples, which corresponds to $4d$ case and $3d$ case respectively.
\end{remark}

\begin{remark}
We point out that for this problem, we handle both of the defocusing case and the focusing case together. However, generally speaking, when one considers long time dynamics of NLS with large data, the focusing case and the defocusing case are quite different. The dynamics of NLS are much richer in the focusing setting. We refer to \cite{BDec4,KM1,KVfocec} regarding some focusing NLS results for the Euclidean case. 
\end{remark}

We make some more comments on previous related results. When $n=0$, the NLS problems are in the setting of pure Euclidean space, which are different from the tori case ($m=0$) or waveguide case. Moreover, defocusing energy critical NLS are well studied in Colliander-Keel-Staffilani-Takakao-Tao \cite{Iteam1} (3d case) and Ryckman-Visan \cite{RV4} (4d case),  so we put them aside. (See also Killip-Visan \cite{3new} and Visan \cite{new4} for new proofs using long time Stricharz estimate technique established in Dodson \cite{BD3}.) Along this paper, we consider the case when $n \geq 1$. Also, we point out that Theorem 1.3 in Killip-Visan \cite{KV1} is the analogue of Theorem \ref{main} for the tori case. For some specific models, global well-posedness and even scattering behavior are expected. Please see \cite{HTT1,HTT2,IPT3,IPRT3,KV1,Haitian,Z1,Z2} for more information. The main purpose of this paper is to give a unified and simpler treatment of well-posedness results based on the Strichartz estimate established in Barron \cite{Barron}.\vspace{3mm}

The proof of Theorem \ref{main} follows from a standard fixed point argument based on a bilinear estimate, which will be explained explicitly in Section 5. The organization of the rest of this paper is: in Section 2, we discuss the preliminaries including notations, Littlewood-Pelay theory and function spaces; in Section 3, we give an overview of the Strichartz estimate on waveguide; in Section 4, we prove some important estimates; in Section 5, we give the proof of the main theorem. 

\section{Preliminaries}
In this section, we discuss notations and function spaces which are initially established by S. Herr, D.Tataru, and N. Tzvetkov \cite{HTT1,HTT2} and have been especially widely applied for NLS problems on tori or waveguides. See \cite{IPT3,IPRT3,KV1,Haitian,Z1,Z2} for examples. Also see Dodson \cite{BD2} for a result which applies similar spaces in the Euclidean setting.\vspace{3mm}

Another thing we want to emphasize here is, for convenience, the following setting in Subsection 2.1 is for the usual tori, i.e, $\mathbb{T}^n=\mathbb{R}^n/\mathbb{Z}^n$. We refer to Killip-Visan \cite{KV1} and Fan-Staffilani-Wang-Wilson \cite{Fan} for the general setting including irrational tori. As we claimed in Remark \ref{1.3}, our result is indifferent to the rational/irrational choice for the tori direction. 
\subsection{Notations}
we write $A \lesssim B$ to say that there is a constant $C$ such that $A\leq CB$. We use $A \simeq B$ when $A \lesssim B \lesssim A $. Particularly, we write $A \lesssim_u B$ to express that $A\leq C(u)B$ for some constant $C(u)$ depending on $u$.\vspace{3mm}

Throughout this paper, we regularly refer to the spacetime norms 
\begin{equation}
    ||u||_{L^p_tL^q_z(I_t \times \mathbb{R}^m\times \mathbb{T}^n)}=\left(\int_{I_t}\left(\int_{\mathbb{R}^m\times \mathbb{T}^n} |u(z)|^q dz \right)^{\frac{p}{q}} dt\right)^{\frac{1}{p}}.
\end{equation}

Now we turn to the Fourier transformation and Littlewood-Paley theory. We define the Fourier transform on $\mathbb{R}^m \times \mathbb{T}^n$ as follows:
\begin{equation}
    (\mathcal{F} f)(\xi)= \int_{\mathbb{R}^m \times \mathbb{T}^n}f(z)e^{-iz\cdot \xi}dz,
\end{equation}
where $\xi=(\xi_1,\xi_2,...,\xi_{d})\in \mathbb{R}^m \times \mathbb{Z}^n$. We also note the Fourier inversion formula
\begin{equation}
    f(z)=c \sum_{(\xi_{m+1},...,\xi_{d})\in \mathbb{Z}^n} \int_{(\xi_1,...,\xi_{m}) \in \mathbb{R}^m} (\mathcal{F} f)(\xi)e^{iz\cdot \xi}d\xi_1...d\xi_m.
\end{equation}
We define the Schr{\"o}dinger propagator $e^{it\Delta}$ by 
\begin{equation}
    \left(\mathcal{F} e^{it\Delta}f\right)(\xi)=e^{-it|\xi|^2}(\mathcal{F} f)(\xi).
\end{equation}
We are now ready to define the Littlewood-Paley projections. First, we fix $\eta_1: \mathbb{R} \rightarrow [0,1]$, a smooth even function satisfying
\begin{equation}
    \eta_1(\xi) =
\begin{cases}
1, \ |\xi|\le 1,\\
0, \ |\xi|\ge 2,
\end{cases}
\end{equation}
and $N=2^j$ a dyadic integer. Let $\eta^d=\mathbb{R}^d\rightarrow [0,1]$, $\eta^d(\xi)=\eta_1(\xi_1)\eta_1(\xi_2)\eta_1(\xi_3)...\eta_1(\xi_d)$. We define the Littlewood-Paley projectors $P_{\leq N}$ and $P_{ N}$ by 
\begin{equation}
    \mathcal{F} (P_{\leq N} f)(\xi):=\eta^d\left(\frac{\xi}{N}\right) \mathcal{F} (f)(\xi), \quad \xi \in \mathbb{R}^m \times \mathbb{Z}^n,
\end{equation}
and
\begin{equation}
P_Nf=P_{\leq N}f-P_{\leq \frac{N}{2}}f.    
\end{equation}
For any $a\in (0,\infty)$, we define
\begin{equation}
    P_{\leq a}:=\sum_{N\leq a}P_N,\quad P_{> a}:=\sum_{N>a}P_N.
\end{equation}

\subsection{Function spaces} In this subsection, we describe the function spaces used in this paper. For $C=[-\frac{1}{2},\frac{1}{2})^d \in \mathbb{R}^d$ and $z\in \mathbb{R}^d$ ($d=3,4$), we denote by $C_z=z+C$ the translate by $z$ and define the sharp projection operator $P_{C_z}$ as follows: ($\mathcal{F}$ is the Fourier transform): 
\[
\mathcal{F}(P_{C_z} f)=\chi_{C_z}(\xi) \mathcal{F} (f)  (\xi).
\]
Here $\chi_{C_z}$ is the characteristic function restrained on $C_z$. We use the same modifications of the atomic and variation space norms. Namely, for $s\in \mathbb{R}$, we define:
\[ \|u\|_{X^s(\mathbb{R})}^2=\sum_{z\in \mathbb{Z}^d} \langle z \rangle^{2s} \|P_{C_z} u\|_{U_{\Delta}^2(\mathbb{R};L^2)}^2
\]
\noindent and similarly we have,  
\[\|u\|_{Y^s(\mathbb{R})}^2=\sum_{z\in \mathbb{Z}^d} \langle z \rangle^{2s} \|P_{C_z} u\|_{V_{\Delta}^2(\mathbb{R};L^2)}^2\]          
\noindent where the $U_{\Delta}^p$ and $V_{\Delta}^p$ are the atomic and variation spaces respectively of functions on $\mathbb{R}$ taking values in $L^2(\mathbb{R}^m\times \mathbb{T}^n)$. There are some nice properties of those spaces. % give the some definitions here. 

\begin{definition}Let $1\leq p < \infty$, and $H$ be a complex Hilbert space. A $U^p$-atom is a piecewise defined function, $a:\mathbb{R} \rightarrow H$，
\[ a=\sum_{k=1}^{K}\chi_{[t_{k-1},t_k)}\phi_{k-1},
\]
where $\{t_k\}_{k=0}^{K} \in \mathcal{Z}$ and $\{\phi_k\}_{k=0}^{K-1} \subset H$ with $\sum\limits_{k=0}^{K}||\phi_k||^p_H=1$. Here we let $\mathcal{Z}$ be the set of finite partitions $-\infty<t_0<t_1<...<t_K\leq \infty$ of the real line.\vspace{3mm}

The atomic space $U^p(\mathbb{R};H)$ consists of all functions $u:\mathbb{R}\rightarrow H$ such that
%
%\noindent                          
 $u=\sum\limits_{j=1}^{\infty}\lambda_j a_j$ for $U^p$-atoms $a_j$, $\{\lambda_j\} \in l^1$, with norm 
 \[||u||_{U^p}:=\inf\left\{\sum^{\infty}_{j=1}|\lambda_j|:u=\sum_{j=1}^{\infty}\lambda_j a_j,\lambda_j\in \mathbb{C}, a_j \text{ is }  U^p\textmd{-atom}\right\}.\]
\end{definition}
\begin{definition}Let $1\leq p < \infty$, and $H$ be a complex Hilbert space. We define $V^p(\mathbb{R},H)$ as the space of all functions $v:\mathbb{R} \rightarrow H$ such that
\[ ||u||_{V^p}:=\sup\limits_{\{t_k\}^K_{k=0} \in \mathcal{Z}}\left(\sum_{k=1}^{K}||v(t_k)-v(t_{k-1})||^p_{H}\right)^{\frac{1}{p}} \leq +\infty,
\]
\noindent where we use the convention $v(\infty)=0$.
\noindent Also, we denote the closed subspace of all right-continuous functions $v:\mathbb{R}\rightarrow H$ such that $\lim\limits_{t\rightarrow -\infty}v(t)=0$ by $V^p_{rc}(\mathbb{R},H)$. 
\end{definition}
For the purpose of this problem, we choose the Hilbert space $H$ to be $L^2$-based Sobolev space $H^s(\mathbb{R}^m\times \mathbb{T}^n)$. 
\begin{definition}For $s\in \mathbb{R}$, we let $U^p_{\Delta}H^s$ resp. $V^p_{\Delta}H^s$ be the spaces of all functions such that $e^{-it\Delta}u(t)$ is in $U^p(\mathbb{R},H^s)$ resp. $V^p_{rc}(\mathbb{R},H)$, with norms
\[||u||_{U^p_{\Delta}H^s}=||e^{-it\Delta}u||_{U^p(\mathbb{R},H^s)}, \quad ||u||_{V^p_{\Delta}H^s}=||e^{-it\Delta}u||_{V^p(\mathbb{R},H^s)}.
\]
\end{definition} For an interval $I \subset \mathbb{R}$, we can also define the restriction norms $X^s(I)$ and $Y^s(I)$ in the natural way:
\noindent                                               $||u||_{X^s(I)}= \inf \{||v||_{X^s(\mathbb{R})}:v\in X^s(\mathbb{R})$ satisfying $v_{|I}=u_{|I}\}$. And similarly for $Y^s(I)$. \vspace{3mm}

Norms $X^s$ and $Y^s$ are both stronger than the $L^{\infty}(\mathbb{R};H^s)$ norm and weaker than the norm $U^2_{\Delta}(\mathbb{R}:H^s)$. Moreover, they satisfy the following property (for $p>2$):
\begin{equation}\label{embd}
    U^2_{\Delta}(\mathbb{R}:H^s) \hookrightarrow X^s\hookrightarrow  Y^s \hookrightarrow V^2_{\Delta}(\mathbb{R}:H^s)\hookrightarrow  U^p_{\Delta}(\mathbb{R}:H^s) \hookrightarrow L^{\infty}(\mathbb{R};H^s).
\end{equation}
Also, we note two useful estimates as follows,
\begin{equation}\label{property1}
   ||u||_{L^{\infty}_{t}H^s_x([0,T)\times \mathbb{R}^m\times \mathbb{T}^n)}\lesssim ||u||_{X^s([0,T))},
\end{equation}
and
\begin{equation}\label{property2}
      \left\|\int_{0}^{t} e^{i(t-s)\Delta} F(s) ds \right\|_{X^s([0,T))} \lesssim ||F(u)||_{L^{1}_{t}H^s_x([0,T)\times \mathbb{R}^m\times \mathbb{T}^n)}.
\end{equation}

\noindent In order to control the nonlinearity on interval $I$, we define `$N$ -Norm' on an interval $I=(a,b)$ corresponding to the Duhamel term as follows,
\begin{equation}
\| h\|_{N^s(I)}=\left\|\int_{a}^{t} e^{i(t-s)\Delta} h(s) ds \right\|_{X^s(I)} .
\end{equation}
\noindent We also have the following proposition.
\begin{proposition}\label{estimateY} If $f\in L^1_t(I,H^1(\mathbb{R}^m\times \mathbb{T}^n))$, then 
\[||f||_{N(I)} \lesssim \sup_{\substack{ v\in Y^{-1}(I),\\
||v||_{Y^{-1}(I)}\leq 1}} \int_{I \times (\mathbb{R}^m\times \mathbb{T}^n)} f(t,x)\overline{v(t,x)}dxdt.
\]
Also, we have the following estimate holds for any smooth function $g$ on an interval $I=[a,b]$: 
\[ ||g||_{X^1(I)}\lesssim ||g(0)||_{H^1(\mathbb{R}^m\times \mathbb{T}^n)}+\left(\sum_N ||P_N(i\partial_t+\Delta)g||^2_{L^1_t(I,H^1(\mathbb{R}^m\times \mathbb{T}^n))}\right)^{\frac{1}{2}}.
\]
\end{proposition}
The proof of Proposition \ref{estimateY} is standard so we omit it. We refer to Proposition 2.11 of Herr-Tataru-Tzvetkov \cite{HTT1} for more details.\vspace{3mm}

Along this paper, we almost always take $I=[0,T)$. Also, we note that
\begin{equation}\label{def}
    X_c^1=X^1\cap C^0_tH^1_x.
\end{equation}
\section{Overview of Strichartz estimate for waveguide manifold}
In this section, we give an overview of the Strichartz estimate for the Schr\"odinger equation on waveguide. For Euclidean case (when $n=0$), there is a classical Strichartz estimate as follows. (See Tao's book \cite{Taobook} for example)
\begin{theorem}[Euclidean Strichartz estimate] \label{Euclidean}
Suppose $\frac{2}{q}+\frac{d}{p}=\frac{d}{2}$, where $p,q \geq 2$ and $(q,p,d)\neq (2,\infty,2)$. Then
\begin{equation}
    ||e^{it\Delta_{\mathbb{R}^d}}f||_{L^q_tL^p_x} \lesssim ||f||_{L^2(\mathbb{R}^d)}.
\end{equation}
\end{theorem}
Over the last few decades, there has been a wide range of research concerning Strichartz estimates for dispersive equations on manifolds other than Euclidean space, in particular in the case of tori or more generally a compact Riemannian manifold. The study of Strichartz estimates for NLS on tori dates back to work of Bourgain \cite{Bourgain1}, and it is only very recently that the full range of (essentially sharp) $L^p_{t,x}$
local estimates have been proved as a corollary of Bourgain and Demeter's decoupling theorem \cite{BD}. See also the work of Killip and Visan \cite{KV1}, which sharpens Bourgain and Demeter’s Strichartz estimate.\vspace{3mm} 

For waveguide case, we recall the Strichartz estimate (local-in-time version) proved by Barron \cite{Barron} as follows. The proof of Theorem \ref{Stricharz} is based on the decoupling method established in Bourgain-Demeter \cite{BD}.

\begin{theorem}[Waveguide Strichartz estimate]\label{Stricharz}
For any bounded time interval $I$ and $p \geq \frac{2(d+2)}{d}$ (where $d=m+n$), one has
\begin{equation}
    ||e^{it\Delta_{\mathbb{R}^m\times \mathbb{T}^n}} P_{\leq N} f||_{L^p(I \times \mathbb{R}^m\times \mathbb{T}^n)} \lesssim_{I,\epsilon}  N^{\epsilon+\frac{d}{2}-\frac{d}{p}}||f||_{L^2},
\end{equation}
where the loss of $N^{\epsilon}$ can be removed for $p$ away from the endpoint.
\end{theorem}

\begin{remark}
There is also a global version of Strichartz estimate in Barron \cite{Barron} (see Theorem 1.1 of Barron \cite{Barron}) which is useful to develop large data long time dynamics of Schr\"odinger initial value problem \eqref{maineq2}. We expect to apply it to obtain long time behavior of NLS on waveguide or tori. It seems that it is not convenient to give a unified and simpler treatment of the long time theory since different specific models have different structures which causes differences for `profile decomposition argument' and `rigidity argument'. (See Kenig-Merle's series work \cite{KM1,KM2} for `Concentration compactness/Rigidity method'.) For the purpose of this paper, the local version is enough.
\end{remark}

\begin{remark}
Strichartz-type estimates has been widely studied and applied for NLS problems on waveguide or tori. We refer to \cite{Guo,HTT1,HTT2,IPT3,IPRT3,KV1,Haitian,Z1,Z2} for related results. It is also possible to obtain Strichartz-type estimates for waveguide case by using `predecoupling techniques' such as Hardy-Littlewood circle method. See Theorem 3.1 in Hani-Pausader \cite{HP} and Proposition 2.1 in Ionescu-Pausader \cite{IPRT3} as examples. But it is hard to get the sharp thresholds and give a unified proof for all of the waveguide models without decoupling method.  Finally, we may also compare Theorem \ref{Stricharz} with Theorem 1.1 in \cite{KV1} which is for the tori case. It seems that the statement still holds when we may replace tori by waveguide with the same whole dimension.
\end{remark}

The following estimate is implied by Strichartz estimate using the properties of the function spaces \eqref{embd}.
\begin{lemma}
According to the atomic structure of $U^p$ and Strichartz estimate (Theorem \ref{Stricharz}), we have

    \begin{equation}\label{Lm3.5.1}
    ||P_{\leq N} f||_{L^p([0,T) \times \mathbb{R}^m\times \mathbb{T}^n)} \lesssim  N^{\frac{d}{2}-\frac{d}{p}}||P_{\leq N}||_{U^p_{\Delta}L^2}\lesssim N^{\frac{d}{2}-\frac{d}{p}}||P_{\leq N}||_{Y^0([0,T))},
\end{equation}
where $p > \frac{2(d+2)}{d}$ and $N\geq 1$. In particular, due to the Galilean invariance of solutions to the linear Schr{\"o}dinger equation, we have
\begin{equation}\label{Lm3.5.2}
    ||P_{C} f||_{L^p([0,T) \times \mathbb{R}^m\times \mathbb{T}^n)} \lesssim N^{\frac{d}{2}-\frac{d}{p}}||P_{C} f||_{Y^0([0,T))},
\end{equation}
for all $p > \frac{2(d+2)}{d}$ and for any cube of $C\subset \mathbb{R}^d$ side-length $N\geq 1$.
\end{lemma}

\section{Some estimates}
In this section, we establish some important estimates which are essential for proving well-posedness results. We start with the following bilinear estimates.

\begin{lemma}[Bilinear estimate]\label{bilinear}
Fix whole dimension $d \geq 3$ and $0<T<1$. Then for $1\leq N_2 \leq N_1$, we have
\begin{equation}\label{bilinear}
    ||u_{N_1} v_{N_2}||_{L^2_{t,x}([0,T)\times \mathbb{R}^{m} \times \mathbb{T}^{n} )}\lesssim N_2^{\frac{d-2}{2}}||u_{N_1}||_{Y^0([0,T))} ||v_{N_2}||_{Y^0([0,T))}.
\end{equation}
\end{lemma}
\begin{proof}
The proof is similar to Lemma 3.1 in Killip-Visan \cite{KV1} which is for the tori case. First, we decompose $\mathbb{R}^d=\cup_j C_j$ where each $C_j$ is a cube of side length $N_2$. We note $P_{C_j}$ for the Fourier projection onto the cube $C_j$. As the spatial Fourier support of $(P_{C_j}u_{N_1})v_{N_2}$ is contained in a fixed dilate of the cube $C_j$, for each $j$, we deduce that 
\begin{equation}
    ||u_{N_1} v_{N_2}||_{L^2_{t,x}([0,T)\times \mathbb{R}^{m} \times \mathbb{T}^{n} )}\lesssim \left(\sum_j ||(P_{C_j}u_{N_1}) v_{N_2}||^2_{L^2_{t,x}([0,T)\times \mathbb{R}^{m} \times \mathbb{T}^{n} )}\right)^{\frac{1}{2}}.
\end{equation}
Applying the Strichartz inequality , we estimate
\begin{equation}
    \aligned
    ||(P_{C_j}u_{N_1}) v_{N_2}||^2_{L^2_{t,x}([0,T)\times \mathbb{R}^{m} \times \mathbb{T}^{n} )}&\lesssim ||P_{C_j}u_{N_1} ||_{L^4_{t,x}([0,T)\times \mathbb{R}^{m} \times \mathbb{T}^{n} )}|| v_{N_2}||_{L^4_{t,x}([0,T)\times \mathbb{R}^{m} \times \mathbb{T}^{n} )}\\
    &\lesssim N_2^{\frac{d-2}{2}}||P_{C_j}u_{N_1}||_{Y^0([0,T))} ||v_{N_2}||_{Y^0([0,T))}.
    \endaligned
\end{equation}
Noticing that
\begin{equation}
    ||u_{N_1}||_{Y^0([0,T))}=\left(\sum_j ||P_{C_j}u_{N_1}||^2_{Y^0([0,T))} \right)^{\frac{1}{2}},
\end{equation}
we can derive \eqref{bilinear}. The proof of Lemma \ref{bilinear} is now complete.
\end{proof}
\begin{remark}
One has stronger bilinear estimate in the Euclidean setting as follows,
\begin{equation}
    ||u_{N_1} v_{N_2}||_{L^2_{t,x}([0,T)\times \mathbb{R}^{d} )}\lesssim N_1^{-\frac{1}{2}}N_2^{\frac{d-1}{2}}||u_{N_1}||_{Y^0([0,T))} ||v_{N_2}||_{Y^0([0,T))}.
\end{equation}
It is not expected such estimate holds on domain involving the torus. However, it does hold up to the loss of $N_2^{\epsilon}$ for tori case. See Fan-Staffilani-Wang-Wilson \cite{Fan} for more details. Lemma \ref{bilinear} is enough for the purpose of this paper. 
\end{remark}

Based on the bilinear estimate, we are ready to prove the following estimate which is essential for us to establish the well-posedness theory.
\begin{lemma}\label{mainest}
Fix whole dimension $d=3,4$, for any $0<T<1$, we have

\begin{equation}\label{mainest1}
   \left\|\int_0^t e^{i(t-s)\Delta} F(u(s)) ds  \right\|_{X^1([0,T))} \lesssim ||u||^{\frac{d+2}{d-2}}_{X^1([0,T))},
\end{equation}
and 
\begin{equation}\label{mainest2}
\aligned
    &\left\|\int_0^t e^{i(t-s)\Delta}  \left( F(u+w)(s)-F(u)(s) \right) ds  \right\|_{X^1([0,T))} \\
    &\lesssim ||w||_{X^1([0,T))}\left(||u||_{X^1([0,T))}+||w||_{X^1([0,T))}\right)^{\frac{4}{d-2}}.
    \endaligned
\end{equation}

\end{lemma}
\begin{proof}
Obviously estimate \eqref{mainest2} implies \eqref{mainest1} in view of taking $u\equiv 0$. It suffices to prove \eqref{mainest2}. The three dimensional case and the four dimensional case will discussed respectively. The proof has the same spirit to Proposition 4.1 in \cite{KV1} which is for the tori case. We will give the sketch of the proof as follows.\vspace{3mm}

By duality and Proposition \ref{estimateY}, 
\begin{equation}
    \aligned
     &\left\|\int_0^t e^{i(t-s)\Delta} P_{\leq N} \left( F(u+w)(s)-F(u)(s) \right) ds  \right\|_{X^1([0,T))} \\
     \leq 
 &  \sup_{\substack{ v\in Y^{-1}(I),\\
     ||v||_{Y^{-1}([0,T))}= 1}} \int_{[0,T) \times (\mathbb{R}^m\times \mathbb{T}^n)} P_{\leq N} \left( F(u+w)(s)-F(u)(s) \right) \overline{v(t,z)}dtdz.
    \endaligned
\end{equation}
We will show that,
\begin{equation}
    \aligned
  &\int_{[0,T) \times (\mathbb{R}^m\times \mathbb{T}^n)}  \left( F(u+w)(s)-F(u)(s) \right) \tilde{v}(x,t)dxdt \\ 
  &\leq  ||\tilde{v}||_{Y^{-1}} ||w||_{X^1([0,T))}\left(||u||_{X^1([0,T))}+||w||_{X^1([0,T))}\right)^{\frac{4}{d-2}},
    \endaligned
\end{equation}
where $\tilde{v}=\overline{P_{\leq N}v}$. It is clear that estimate \eqref{mainest2} follows from the above estimate by letting $N \rightarrow \infty$. Moreover, according to some basic combinatorics, we can reduce it to prove the following,
\begin{equation}
    \aligned
     &\sum_{N_0 \geq 1} \sum\limits_{N_1\geq ...\geq N_{\frac{d+2}{d-2}} \geq 1} \left|\int_{[0,T) \times (\mathbb{R}^m\times \mathbb{T}^n)} \tilde{v}_{N_0}\prod\limits_{j=1}^{\frac{d+2}{d-2}}u^{(j)}_{N_j}dxdt\right| \leq  ||\tilde{v}||_{Y^{-1}}\prod\limits_{j=1}^{\frac{d+2}{d-2}}||u^{(j)}||_{X^1([0,T))},\\
    \endaligned
\end{equation}
by choosing $u^{(j)}$ varying over the collection $\{u,\bar{u},w,\bar{w} \}$.\vspace{3mm} 

The rest of the proof follows as Proposition 4.1 in \cite{KV1}, so we omit it. $3d$ case and $4d$ case are discussed respectively.
\end{proof}

\section{Well-posedness for the energy-critical NLS}
In this section, we prove the main result, i.e. Theorem \ref{main} by contraction mapping argument based on Lemma \ref{mainest}. The proof has the same spirit of Theorem 1.3 in \cite{KV1}. We will deal with the local well-posedness argument and the small data global well-posedness argument respectively. Moreover, we give a unified treatment for both of the three dimensional case and the four dimensional case. \vspace{2mm} 

\emph{Proof of Theorem \ref{main}}:\vspace{3mm} 

Part 1 (Small data global well-posedness): We start with the statement in Theorem \ref{main} concerning small initial data. Fix the dimension $d$ and a small initial data $u_0$ satisfying 
\begin{equation}
    ||u_0||_{H^1}<\eta_0,
\end{equation}
where $\eta_0=\eta(d)$ to be chosen later.\vspace{3mm}

Following the classical theory, we note that by conservation of mass and energy, it suffices to construct the solution to the initial-value problem \eqref{maineq2} on the time interval $[0,1]$. According to Sobolev embedding, the full energy $E(u)+M(u)$ is similar to the $H^1$ norm of $u$.\vspace{3mm}

To construct the solution to \eqref{maineq2} on the time interval $[0,1]$, we use a contraction mapping argument. More precisely, we will show that the mapping 
\begin{equation}
    \Phi(u)(t)=e^{it\Delta}u_0\pm i\int_0^T e^{i(t-s)\Delta} F(u(s)) ds
\end{equation}
(which arises from Duhamel's formula) is a contraction on the ball
\begin{equation}
    B:=\left\{u\in X^1_c: ||u||_{X^1([0,1])} \leq 2\eta\right\}
\end{equation}
(where space $X^1_c$ is defined in \eqref{def}), under the metric
\begin{equation}
    d(u,v):=||u-v||_{X^1([0,1])}.
\end{equation}
Using Lemma \ref{mainest}, 
\begin{equation}
    \aligned
||\Phi(u)||_{X^1([0,1])}&\leq ||e^{it\Delta}u_0||_{X^1([0,1])}+ \left\|\int_0^t e^{i(t-s)\Delta} F(u(s)) ds\right\|_{X^1([0,1])} \\
&\leq ||u_0||_{H^1}+C||u||^{\frac{d+2}{d-2}}_{X^1([0,1])} \leq \eta+C(2\eta)^{\frac{d+2}{d-2}} \leq 2\eta ,
    \endaligned
\end{equation}
provided $\eta_0$ is chosen sufficiently small. This proves that $\Phi$ maps the ball $B$ to itself.\vspace{3mm}

Now we prove that $\Phi$ is also a contraction mapping. Again, using Lemma \ref{mainest},
\begin{equation}
    \aligned
d(\Phi(u),\Phi(v)) &\leq   \left\|\int_0^t e^{i(t-s)\Delta} \left( F(u)(s)-F(v)(s) \right)  ds  \right\|_{X^1([0,T))} \\
&\leq C||u-v||_{X^1([0,T))}\left(||u||_{X^1([0,T))}+||v||_{X^1([0,T))}\right)^{\frac{4}{d-2}}\\
&\leq C (4\eta )^{\frac{4}{d-2}}  d(u,v)
\leq \frac{1}{2}d(u,v),
    \endaligned
\end{equation}
provided $\eta_0$ is chosen sufficiently small. This completes the discussion of small initial data argument.\vspace{3mm}

Part 2 (Local well-posedness): Now we turn to the case of large initial data. For this case, we need to use frequency truncation method and deal with the nonlinearity more delicately. First, we let initial data $u_0$ satisfies $||u_0||_{H^1}<E$. Let $\delta$ be a small number to be chosen later and let $N=N(u_0)$ such that 
\begin{equation}
    ||P_{>N}u_0||_{H^1} \leq \delta.
\end{equation}
We will show that the mapping
\begin{equation}
    \Phi(u)(t)=e^{it\Delta}u_0\pm i\int_0^T e^{i(t-s)\Delta} F(u(s)) ds
\end{equation}
is a contraction on the ball
\begin{equation}
    B:=\left\{u\in X^1_c: ||u||_{X^1([0,1])} \leq 2E, ||P_{>N}u||_{X^1([0,1])} \leq 2\delta \right\},
\end{equation}
provided the time $T$ is chosen sufficiently small (depending on $E$, $\delta$, and $N$). First we verify that $\Phi$ maps ball $B$ to itself. By Lemma \ref{mainest}, \eqref{property1}, \eqref{property2} and Bernstein inequality, for $u \in B$, we have
\begin{equation}
    \aligned
    ||\Phi(u)||_{X^1}&\leq ||e^{it\Delta}u_0||_{X^1}+ \left\|\int_0^t e^{i(t-s)\Delta} F(u_{\leq N}(s)) ds\right\|_{X^1}+ \left\|\int_0^t e^{i(t-s)\Delta} [F(u(s))-F(u_{\leq N}(s))] ds\right\|_{X^1}\\
    &\leq ||u_0||_{H^1}+C||F(u_{\leq N})||_{L_t^1H_x^1}+C||u_{>N}||_{X^1} ||u||^{\frac{4}{d-2}}_{X^1}\\
    &\leq E+CT||u_{\leq N}||_{L^{\infty}_{t}H^1_x} ||u_{\leq N}||^{\frac{4}{d-2}}_{L^{\infty}_{t,x}}+C(2\delta)(2E)^{\frac{4}{d-2}}\\
    &\leq E+CTN^2(2E)^{\frac{d+2}{d-2}}+C(2\delta)(2E)^{\frac{4}{d-2}}
    \leq 2E,
    \endaligned
\end{equation}
provided $\delta$ is chosen small enough depending on $E$, and $T$ is chosen small enough depending on $E$ and $N$. We now decompose the nonlinearity by frequency as follows.
\begin{equation}
    F(u)=F_1(u)+F_2(u),\quad \textmd{where} \quad F_1(u)=O\left(u^2_{>N}u^{\frac{6-d}{d-2}}\right), \quad F_2(u)=O\left(u_{\leq N}^{\frac{4}{d-2}}u\right).
\end{equation}
Then, we estimate,
\begin{equation}
    \aligned
    ||P_{>N}\Phi(u)||_{X^1}&\leq ||e^{it\Delta}P_{>N}u_0||_{X^1}+ \left\|\int_0^t e^{i(t-s)\Delta} F_1(u(s)) ds\right\|_{X^1}
    + \left\|\int_0^t e^{i(t-s)\Delta} F_2(u(s)) ds\right\|_{X^1}\\
    &\leq ||P_{>N}u_0||_{H^1}+C||u_{>N}||^2_{X^1} ||u||^{\frac{6-d}{d-2}}_{X^1}+C||F_2(u)||_{L_t^1H_x^1}\\
    &\leq \delta+C(2\delta)^2(2E)^{\frac{6-d}{d-2}}+CT\left(||\nabla u||_{L_t^{\infty}L^2_x}||P_{\leq N}u||^{\frac{4}{d-2}}_{L^{\infty}_{t,x}}+||u||_{L^{\infty}_{t}L_x^{\frac{2d}{d-2}}}N||u_{\leq N}||^{\frac{4}{d-2}}_{L^{\infty}_{t}L^{\frac{4d}{d-2}}_{x}}\right)\\
    &\leq \delta+C(2\delta)^2(2E)^{\frac{6-d}{d-2}}+CTN^2(2E)^{\frac{d+2}{d-2}}
    \leq 2\delta,
    \endaligned
\end{equation}
provided $\delta$ is chosen small enough depending on $E$, and $T$ is chosen small enough depending on $E$, $\delta$, and $N$. Next it remains to show that $\Phi$ is also a contraction. We notice that 
\begin{equation}
    F_1(u)-F_1(v)=O\left((u-v)\left(u_{>N}+v_{>N}\right)\left(u^{\frac{6-d}{d-2}}+v^{\frac{6-d}{d-2}}\right)\right),
\end{equation}
and
\begin{equation}
    F_2(u)-F_2(v)=O\left((u-v)(u_{\leq N}+v_{\leq N})^{\frac{4}{d-2}}\right)+O\left(\left(u_{\leq N}-v_{\leq N}\right)(u+v)\left(u_{\leq N}+v_{\leq N}\right)^{\frac{6-d}{d-2}}\right).
\end{equation}
Employing Lemma \ref{mainest}, \eqref{property1}, \eqref{property2} and the Bernstein inequality, for $u,v\in B$, we estimate,
\begin{equation}
    \aligned
 &   d(\Phi(u),\Phi(v))\\
 & \lesssim ||u-v||_{X^1}(||u_{>N}||_{X^1}+||v_{>N}||_{X^1})(||u||_{X^1}+||v||_{X^1})^{\frac{6-d}{d-2}}+||F_2(u)-F_2(v)||_{L^1_tH^1_x}\\
   &\lesssim (4\delta)(4E)^{\frac{6-d}{d-2}}d(u,v)+T||\nabla (u-v)||_{L_t^{\infty}L_x^{2}}(||u||_{L^{\infty}_{x,t}}+||v||_{L^{\infty}_{x,t}})^{\frac{4}{d-2}}\\
   &\quad + T||u-v||_{L_t^{\infty}L_x^{\frac{2d}{d-2}}}N\left(||u_{\leq N}||_{L_t^{\infty}L_x^{\frac{4d}{d-2}}}+||v_{\leq N}||_{L_t^{\infty}L_x^{\frac{4d}{d-2}}}\right)^{\frac{4}{d-2}}\\
   &\quad + T\left(||\nabla u||_{L^{\infty}_tL^2_x}+||\nabla v||_{L^{\infty}_tL^2_x}\right)(||u_{\leq N}-v_{\leq N}||_{L^{\infty}_{x,t}})\left(||u_{\leq N}||_{L^{\infty}_{x,t}}+||v_{\leq N}||_{L^{\infty}_{x,t}}\right)^{\frac{6-d}{d-2}}\\
   & \quad + T\left(||u||_{L^{\infty}_tL_x^{\frac{2d}{d-2}}}+||v||_{L^{\infty}_tL_x^{\frac{2d}{d-2}}}\right)N||u_{\leq N}-v_{\leq N}||_{L^{\infty}_tL^{\frac{4d}{d-2}}_x}  \left(||u_{\leq N}||_{L_t^{\infty}L_x^{\frac{4d}{d-2}}}+||v_{\leq N}||_{L_t^{\infty}L_x^{\frac{4d}{d-2}}}\right)^{\frac{6-d}{d-2}}\\
   &\lesssim \left( 4\delta (4E)^{\frac{6-d}{d-2}}+TN^2(4E)^{\frac{4}{d-2}} \right)d(u,v)
   \leq \frac{1}{2}d(u,v),
    \endaligned
\end{equation}
provided $\delta$ is chosen small enough depending on $E$, and $T$ is chosen small enough depending on $E$ and $N$.\vspace{3mm}

At last, we make comments on the uniqueness. By the standard contraction mapping theorem, the above analysis allows us to construct a unique solution
$u$ to \eqref{maineq2} in the ball $B$. To see that uniqueness holds in the larger class $X^1([0, T ]) \cap C^0_tH^1_
{x}([0, T]\times \mathbb{R}^{m} \times \mathbb{T}^{n})$, it suffices to observe that if $v \in X^1([0, T ]) \cup C^0_tH^1_
{x}([0, T])$ is a second solution to \eqref{maineq2} with data $v(0) =u_0$, then there exists $N_0 \geq 1$ such that
\begin{equation}
    ||v_{>N_0}||_{X^1([0, T ])}<2\delta.
\end{equation}
Choosing the larger of $N$ and $N_0$, we find a new ball $B$ that contains both $u$ and $v$. In this way, the contraction mapping argument guarantees $u=v$ on a possibly smaller interval $[0,T{'}]$. Iterating this argument yields uniqueness in the larger class.\vspace{3mm}

The proof of Theorem \ref{main} is now complete.\vspace{5mm}

\noindent \textbf{Acknowledgments.}  We highly appreciate Professor Benjamin Dodson and Professor Changxing Miao for their kind suggestions, help and discussions. Xing Cheng has been partially supported by the NSFC (No.11526072). Zehua Zhao is partially supported by University of Maryland (postdoctoral research support). Jiqiang Zheng was partially supported by the NSFC under grants 11771041, 11831004, 11901041. Zehua Zhao was also a guest of Institute of Applied Physics and Computational Mathematics during the writing of this paper.

 \bibliographystyle{amsplain}

\noindent \author{Xing Cheng}

\noindent \address{College of Science, Hohai University}\\
{Nanjing 210098, Jiangsu, China.}

\noindent \email{chengx@hhu.edu.cn}\\

\noindent \author{Zehua Zhao}

\noindent \address{Department of Mathematics, University of Maryland}\\
{William E. Kirwan Hall, 4176 Campus Dr.
College Park, MD 20742-4015.}

\noindent \email{zzh@umd.edu}\\

\noindent \author{Jiqiang Zheng}

\noindent \address{Institute of Applied Physics and Computational Mathematics}\\
{Beijing, 100088, P.R.China.}

\noindent \email{zhengjiqiang@gmail.com}\\

\end{document}